\documentclass[12pt]{article}
\usepackage{soul}
\usepackage[utf8]{inputenc}
\usepackage{hyperref}
\usepackage[english]{babel}
\usepackage{cancel}
\usepackage{amsthm,mathtools,amssymb,color}
\usepackage[right]{showlabels}
\usepackage{rotating}
\renewcommand{\showlabelsetlabel}[1]

\newtheorem{claim}{Claim}
\newtheorem{thm}{Theorem}
\newtheorem{rema}{Remark}

\newtheorem{lemma}{Lemma}
\newtheorem{defn}{Definition}
\def\P{\mathbb{P}}
\def\R{\mathbb{R}}
\def\Z{\mathbb{Z}}
\def\E{\mathbb{E}}
\def\F{\mathcal{F}}

\def\g{\gamma}
\def\a{\alpha}
\def\l{\lambda}
\def\b{\beta}

\def\eps{\varepsilon}

\def\dd{\,\mathrm{d}}
\newcommand{\card}[1]{\mathrm{card}\left(#1\right)}

\title{Transience of continuous-time conservative random walks}

\author{Satyaki Bhattacharya\footnote{Centre for Mathematical Sciences, Lund University, Box 118 SE-22100, Lund, Sweden.} \ and Stanislav Volkov${}^*$}

\begin{document}
\maketitle
\begin{abstract}
We consider two continuous-time generalizations of conservative random walks introduced in~\cite[Englander and Volkov (2022)]{EV}, an orthogonal and a spherically symmetrical one; the latter model is also known as {\em random flights}. For both models, we show the transience of the walks when $d\ge 2$ and the rate of changing of direction follows power law $t^{-\a}$, $0<\a\le 1$, or the law $(\ln t)^{-\beta}$ where $\beta>2$.
\end{abstract}

\noindent
{\bf AMS subject classification}: 60G50, 60J05, 60J75

\noindent
{\bf Keywords}: Random flight, non-time-homogeneous Markov chain, conservative random walk, transience, recurrence

\section{Introduction}
Conservative random walk (discrete time) was introduced in~\cite{EV} as time-nonhomogeneous Markov chain $X_n$, $n=1,2,\dots$,  defined as a process on $\Z^d$ such that for given (non-random) sequence $p_1,p_2,\dots$, where all $p_i\in(0,1)$, the walk at time $n$ with probability $p_n$  randomly picks one of the $2d$ directions parallel to the axis, and otherwise it continues moving in the direction it was going before. This walk can be viewed as generalization of Gillies random walk~\cite{Gillis} A special interesting case studied in~\cite{EV} was when $p_n\to 0$, and in particular when $p_n\sim n^{-\a}$ where $\a\in(0,1]$;
note that the case $\a>1$ is trivial as the walk would make only finally many turns. The question of recurrence vs.\ transience of this walk was one of the main questions of that paper.

Similar processes were known in the literature under different names. Some of the earliest papers which mention the continuous process with memory of this type were probably~\cite{Goldstein,Kac}. The term {\em persistent random was} used in~\cite{Peggy1,Peggy2}; in these papers some very general criteria of recurrence vs.\ transience were investigated. A planar motion with just three directions was studied in~\cite{Di}. A book on {\em Markov random flights} was recently published by Kolesnik~\cite{Kolesnik}. {\em Planar random motions with drifts} with four directions/speeds, switching at Poisson  times, were studied in~\cite{OR}. Applications of telegraph processes to option pricing can be found in~\cite{RAT}. Characteristic functions of correlated random walks were studied in~\cite{Chen}.

The main difference between the conservative random walk and most of the models studied in the literature (except, perhaps, \cite{VAS}, which has a more applied focus) is that the underlying process of switching the direction is {\em time-inhomogeneous}, thus creating various new phenomena. The recurrence/transience of discrete space conservative random walk on $\Z^1$ was thoroughly studied in~\cite{EV1} (please also see the references therein), and we believe that the continuous time version in one dimension will have very similar features. The establishment of recurrence in one dimension is more or less equivalent to finding that $\limsup$ of the process is $+\infty$ and $\liminf$ of the process is $-\infty$, while in higher dimensions the situation is much more intricate. Hence, we concentrate on the case when the dimension of the space is at least 2, except Theorem~\ref{th:A1} which deals with the embedded process.

\vskip 1cm

Below we introduce formally the two versions of a {\em continuous-time} conservative random walk on $\R^d$, $d\ge 1$.
\\[5mm]
\noindent
{\bf Model A (orthogonal model).} Let $\l(t)$ be a non-negative function such that 
\begin{align}\label{eq:conds}
\begin{split}
\Lambda(T)&=\int_0^T \l(t)\dd t<\infty    \quad \text{for all }T\ge 0;\\
\lim_{T\to \infty} \Lambda(T)&=\infty. 
\end{split}
\end{align}
Let $\tau_1<\tau_2<\dots$ be the consecutive points of a non-homogeneous Poisson point process (PPP for short) on $[0,\infty)$ with rate $\l(t)$, and $\tau_0=0$. Then conditions~\eqref{eq:conds} will guarantee that there will be finitely many $\tau_i$'s in every finite interval and that $\tau_n\to\infty$.  Let ${\bf f}_0,{\bf f}_1,{\bf f_2}$ be an i.i.d.\ sequence of vectors, each of which has a uniform distribution on the set of $2d$ unit vectors $\{\pm {\bf e}_1,\pm {\bf e}_2,\dots, \pm {\bf e}_d\}$ in $\R^d$.

The (orthogonal, continuous-time) conservative walk generated by the rate function $\l(\cdot)$ is a process $Z(t)$, $t\ge 0$, in $\R^d$, $d\ge 1$, such that $Z(0)=0$ and at each time $\tau_k$, $k\ge 0$, the walk starts moving in the direction ${\bf f}_k$, and keeps moving in this direction until time $\tau_{k+1}$, when it updates its direction. Formally, define
$$
N(t)=\sup\{k\ge 0:\ \tau_k\le t\},
$$
as the number of points of the PPP by time $t$, then
$$
Z(t)=\sum_{k=0}^{N(t)-1} \left(\tau_{k+1}-\tau_{k}\right){\bf f}_{k} +\left(t-\tau_{N(t)}\right){\bf f}_{N(t)}.
$$
We can also define the embedded process $W_n=Z(\tau_n)$ so that $W(0)=0$ and for $n\ge 1$
$$
W_n=\sum_{k=0}^{n-1} \left(\tau_{k+1}-\tau_{k}\right){\bf f}_{k}.
$$
The process $Z(t)$ can be viewed as a continuous equivalent of the conservative random walk introduced in~\cite{EV}.
\\[5mm]
\noindent
{\bf Model B (von Mises–Fisher model).} This model is defined similarly to the previous one, except that now random vectors ${\bf f}_k$, $k=1,2,\dots$, have a uniform distribution on a $d-$dimensional unit sphere $\mathcal{S}^{d-1}$, often called {\it von Mises–Fisher distribution}, instead of just on $2d$ unit vectors of $\R^d$.

Please note that this model is similar to the ``random flights'' model studied e.g.\ in~\cite{OG}, however, their results are only for a time-homogeneous Poisson process, unlike our case.
\vskip 5mm

The results that we obtain in the current paper are somewhat different for the two models, however, since they share a lot of common features, certain statements will hold for both of them. The main goal is establishing transience vs. recurrence of the walks, defined as follows.
\begin{defn}
Let $\rho\ge 0$. We say that the walk $Z(t)$ is $\rho-$recurrent, if there is an infinite sequence of times $t_1<t_2<\dots$ converging to infinity, such that $Z(t_i)\in [-\rho,\rho]^d$ for all $i=1,2,\dots$.

We say that the walk $Z(t)$ is transient, if it is not $\rho-$recurrent for any $\rho>0$, i.e.\ equivalently $\lim_{t\to\infty} \Vert Z(t)\Vert=\infty$.

Recurrence and transience of the embedded process $W_n$ are defined analogously, with the exception that instead of $t_1,t_2,t_3,\dots$ in the above definition, we have a strictly increasing sequence of positive {\em integers} $n_i$, $i=1,2,\dots$.
\end{defn}

\begin{rema} 
Note that {\em a priori} it is unclear if transience and recurrence are zero-one events, neither can we easily rule out a possibility of ``intermediate'' situations (e.g.\ $\rho-$recurrence  only for {\em some} $\rho$).
\end{rema}

Our main result, which shows transience for two type of rates, are presented in Theorems~\ref{th:A1}, \ref{thm2}, \ref{th:sat} and~\ref{thm:CV}.

\section{Preliminaries}
Throughout the paper, we will use the following notations. We write $X\sim Poi(\mu)$ when $X$ has a Poisson distribution with parameter $\mu>0$. For any set $A$, $|A|$ denotes the cardinality of the set $A$, and for $x\in\R^d$, $\Vert x\Vert$ denotes the usual Euclidean norm of $x$.

First, we state Kesten's generalization of the Kolmogorov-Rogozin inequality. Let $S_n=\xi_1+\dots+\xi_n$ where $\xi_i$ are independent, and for any random variable $Y$ define $Q(Y;a)=\sup_x \P(Y\in[x,x+a])$.

\begin{lemma}[\cite{KE}]\label{lem:kesten}
There exists $C>0$ such that for any real numbers $0<a_1,\dots,a_n\le 2L$ one has
$$
Q(S_n;L)\le\frac{C L\sum_{i=1}^n a_i^2(1-Q(\xi_i;a_i))Q(\xi_i;a_i)}{\left[\sum_{i=1}^n a_i^2 (1-Q(\xi_i;a_i))\right]^{3/2}}
$$
\end{lemma}

Secondly, if $D_1,\dots, D_m$ is a sequence of independent events each with probability $p$, and $\eps>0$,  and $N_D(m)=\card{\{i\in\{1,\dots,m\}:\ D_i\text{ occurs}\}}=\sum_{i=1}^m {\bf 1}_{D_i}$, then
\begin{align}\label{LDP}
\P(\left|N_D(m)-pm\right|\ge \eps m)&\le 2\exp\left(-2\eps^2 m\right);
\end{align}
by Hoeffding's inequality (see e.g.~\cite{Hoef}).

Suppose we have an inhomogeneous PPP with rate 
\begin{align}\label{rate1}
\l(t)=\frac{1}{t^\a},\qquad t>0,    
\end{align}
where $0<\a<1$ is constant; thus $\Lambda(T)=\int_0^T\l(t)\dd t=\frac{T^{1-\a}}{1-\a}$ and conditions~\eqref{eq:conds} are fulfilled. Let $0<\tau_1<\tau_2<\dots$ denote the points of the PPP in the increasing order.

The following statement is probably known, but for the sake of completeness, we provide its short proof.
\begin{claim}\label{clPoi}
Let $Z$ be a Poisson random variable with rate $\mu>0$. Then
\begin{align*}
  \P\left(Z\ge \frac{3\mu}2\right)&\le e^{-\frac{3\ln(3/2)-1}2 \mu}=e^{-0.108\dots \mu},\\
\P\left(Z\le\frac{\mu}2\right)&\le 
e^{-\frac{1-\ln 2}2 \mu}=e^{-0.153\dots \mu}.
\end{align*}
\end{claim}
\begin{proof}
By Markov inequality, since $\E e^{uZ}=e^{\mu(e^u-1)}$, we have for $u>0$
\begin{align*}
\P\left(Z\ge \frac{3\mu}2\right)
\le
\P\left(e^{uZ}\ge e^{\frac{3\mu u}2}\right)
\le e^{-\frac{3\mu u}2} \E e^{uZ}
=e^{-\mu\left[\frac{3u}{2}-e^u+1\right]}.
\end{align*}
Setting $u=\ln(3/2)$ yields the first inequality in the claim.

For the second inequality, we use
\begin{align*}
\P\left(Z\le \frac{\mu}2\right)
\le
\P\left(e^{-uZ}\ge e^{-\frac{\mu u}2}\right)
\le e^{\frac{\mu u}2} \E e^{-uZ}
=e^{-\mu\left[-\frac{u}2-e^{-u}+1\right]}.
\end{align*}
Now let $u=\ln 2$.
\end{proof}

\begin{lemma}\label{lemlarg}
Suppose that the rate of the PPP is  given by~\eqref{rate1}.  For some $c_1>c_0>0$, depending on $\a$ only, we have 
\begin{align*}
\P\left(\tau_k \le c_0 k^{\frac{1}{1-\alpha}}\right)&\le e^{-k/15};
\qquad
\P\left(\tau_k \ge c_1 k^{\frac{1}{1-\alpha}}\right)\le e^{-k/15}. 
\end{align*}
\end{lemma}

\begin{proof}
Recall that $N(s)$ denote the number of points of the PPP by time $s$. Then $N(s)\sim Poi(\Lambda(s))$ and 
$$
\P(\tau_k\le s)=\P(N(s)\ge  k).
$$
Let $T_n=\Lambda^{(-1)}(n)=\sqrt[1-\a]{(1-\a)n}$.
Noting that $N(T_n)\sim Poi(n)$ for all $n$,  we have
$$
\P(\tau_k\le T_{2k/3})=\P\left(N(T_{2k/3})\ge k\right)
=\P\left(Z\ge \frac32\mu \right)\le e^{-0.07213\dots k}.
$$
by Claim~\ref{clPoi}, with $Z\sim Poi(\mu)$ where $\mu=2k/3$.
Similarly,
\begin{align*}
\P(\tau_k\ge T_{2k})=\P(N(T_{2k})\le k)
\le \P\left(Z\le \frac12 \mu\right)
\le e^{-0.30685\dots k}.
\end{align*}
by Claim~\ref{clPoi}, with $Z\sim Poi(\mu)$ where $\mu=2k$.
Note that $0.07213>\frac1{15}$, $0,30685>\frac1{15}$.

Now the statement follows with $c_0=\sqrt[1-\a]{2(1-\a)/3}$ and $c_1=\sqrt[1-\a]{2(1-\a)}$.
\end{proof}

\section{Analysis of Model A}
\begin{thm}\label{th:A1} Let $d=1$, $\a\in(1/3,1)$, and
the rate is given by~\eqref{rate1}. Then the embedded walk $W_n$ is  transient~a.s.
\end{thm}

\begin{proof}
Assume w.l.o.g.\ that $n$ is even. We will show that for any $\rho>0$ the walk $W_n$ visits $[-\rho,\rho]$ finitely often a.s. With probabilities close to one, both events 
\begin{align}\label{eq:EE}
\begin{split}
\mathcal{E}_1&=\left\{\tau_{n/2}\ge c_0 (n/2)^{\frac{1}{1-\alpha}}\right\},\\
\mathcal{E}_2&=\left\{\tau_{n}\le c_1 n^{\frac{1}{1-\alpha}}\right\}
\end{split}
\end{align}
occur; indeed,
\begin{align}\label{eq:Ac}
  \P(\mathcal{E}_1^c)\le  e^{-n/30}  ,\quad
  \P(\mathcal{E}_2^c)\le  e^{-n/15}  
\end{align}
by Lemma~\ref{lemlarg}.

Since  rate $\l(t)=t^{-\a}$ of the Poisson process is monotonously decreasing, random variables $\tau_i - \tau_{i-1}$, $n/2\le i\le n$, under the condition stated in the event $\mathcal{E}_1$, are stochastically larger than i.i.d.\ exponential random variables 
$\zeta_i$ with rates equal to $[c_0(n/2)^{\frac1{1-\a}}]^\a=\tilde c_0\, n^{-\frac{\alpha}{1-\alpha}}$ for some $\tilde c_0>0$.
For $\zeta$, there exists $\beta=\beta(c_0,\a)>0$ such that
$$
\P\left(\zeta_i>\beta n^{\frac{\alpha}{1-\alpha}}\right)=\exp\left(-\tilde c_0\, n^{-\frac{\alpha}{1-\alpha}}\cdot \beta n^{\frac{\alpha}{1-\alpha}}\right)=e^{-\tilde c_0 \beta}=\frac{2}{3}.
$$
Let
$$
I_n=\left\{i\in[n/2,n]:\ \text{$i$ is even}, \ \tau_i-\tau_{i-2}>\beta n^{\frac{\alpha}{1-\alpha}}\right\},
$$
note that $\card{I_n}\le n/4$.
Then, since $\tau_i-\tau_{i-2}>\tau_i-\tau_{i-1}$, by stochastic monotonicity and Hoeffding's inequality~\eqref{LDP} with $m=\frac n4$, $p=\frac23$, and $\eps=\frac23-\frac12=\frac16$
\begin{align}\label{eq:IA}
\P\left(\card{I_n}<\frac 12\cdot \frac n4\mid \mathcal{E}_1 \right)\le 
\P\left(
\card{\left\{i\in[n/2,n]:\ \text{$i$ is even}, \ \zeta_i>\beta n^{\frac{\alpha}{1-\alpha}}
\right\}}<\frac n8\right)\le 
2 e^{-n/72}.
\end{align}
Let
$$
J_n=\left\{i\in I_n:\ {\bf f}_{i-1}=-{\bf f}_{i-2}\right\}.
$$
Since ${\bf f}_i$ are i.i.d.\ and independent of $\{\tau_1,\tau_2,\dots\}$, and $\P({\bf f}_{i-1}=-{\bf f}_i)=1/2$, on the event $\{\card{I_n}\ge n/8\}$ we have by~\eqref{LDP} with $m=\card{I_n}$, $p=1/2$, $\eps=1/34$
\begin{align}\label{eqJn}
\begin{split}
\P\left(\card{J_n}<\frac{n}{17}\mid \card{I_n}\ge \frac n8\right)
&\le 
\P\left(\left|\card{J_n}-\frac{\card{I_n}}2\right|>\frac{\card{I_n}}{34}\mid \card{I_n}\ge \frac n8\right)
\\ &
\le 2 \exp\left\{-2\cdot \frac1{34^2}\cdot \frac n8\right\}=2e^{-c_* n}
\end{split}
\end{align}
where $c_*=\frac 1{4624}$.

Our proof will rely on conditioning over the {\em even stopping times}, that is, on the event $$
\mathcal{D}=\mathcal{D}_0\cap \mathcal{D}_1\cap \mathcal{E}_2,
$$
where
\begin{align}\label{eqD0}
\mathcal{D}_0&=\{ \tau_{n/2}=t_{n/2},\tau_{n/2+2}=t_{n/2+2},\dots,\tau_n=t_n\},
\\ \nonumber
\mathcal{D}_1&=\left\{\card{J_n}\ge n/17\right\}\cap \mathcal{E}_1
\end{align}
for some strictly increasing sequence $0<t_{n/2}<t_{n/2+2}<\dots<t_n$. 
Note that
\begin{align}\label{eqD1}
\begin{split}
\P(\mathcal{D}_1^c)
&
\le  \P\left(\mathcal{E}_1^c\right)+\P\left(\card{J_n}< \frac{n}{17}\mid  \mathcal{E}_1\right)
\\
& \le 
\P\left(\mathcal{E}_1^c\right)
+
\P\left(\card{J_n}< \frac{n}{17}\mid \card{I_n}\ge \frac n8, \mathcal{E}_1\right)
+
\P\left(\card{I_n}< \frac n8\mid  \mathcal{E}_1\right)
\\
&\le 
e^{-n/30}+2 e^{-c_* n}+2e^{-n/72}
\le 5 e^{-c_*n}
\end{split}
\end{align}
by~\eqref{eq:Ac}, \eqref{eq:IA}, and \eqref{eqJn}.

We denote the $i^{th}$ step of the embedded walk by $X_i=W_i-W_{i-1}$, $i=1,2,\dots$, and
$$
\xi_{i}:=X_{i-1}+X_{i}=(\tau_{i-1}-\tau_{i-2}){\bf f}_{i-2}+(\tau_{i}-\tau_{i-1}){\bf f}_{i-1}.
$$ 
Conditioned on $\mathcal{D}$, random variables $\xi_2,\xi_4,\xi_6,\dots$ are then independent.

\begin{lemma}\label{lemm1}
For $i\in J_n$ we have
$$
\sup_{x\in\R} \P \left(
\xi_i\in[x-\rho,x+\rho]\mid \mathcal{D}
\right)
\leq\ \frac{c\rho}{n^\frac{\alpha}{1-\alpha}}
$$ 
for some $c=c(\a)>0$.
\end{lemma}
\begin{proof}[Proof of Lemma~\ref{lemm1}]
Given $\tau_{i-2}=t_{i-2}$ and $\tau_i=t_i$, $\tau_{i-1}$ has the distribution of the only point of the PPP on $[t_{i-2},t_i]$ with rate $\l(t)$ conditioned on the fact that there is exactly one point in this interval. Hence the conditional density of $\tau_{i-1}$ is given by
$$
f_{\tau_{i-1}\mid \mathcal{D}}(x)=\begin{cases}
    \frac{\l(t)}{\int_{t_{i-2}}^{t_i} \l(u)\dd u}, &\text{if }x\in[t_{i-2},t_i];\\
    0,&\text{otherwise.}
\end{cases}
$$
Assume w.l.o.g.\ that $X_{i-1}>0>X_i$ (recall that $i\in J_n$).
Then
$$
\xi_i=[\tau_{i-1}-\tau_{i-2}]-[\tau_{i}-\tau_{i-1}]=2\tau_{i-1}-(t_{i-2}+t_i)
$$
so that the maximum of the conditional density of $\xi_i$ equals one-half of the maximum of the conditional density of $\tau_{i-1}$.
At the same time, since $\l$ is a decreasing function,
\begin{align*}
\sup_{x\in\R} f_{\tau_{i-1}\mid \mathcal{D}}(x)&=\frac{\l(t_{i-2})}{\int_{t_{i-2}}^{t_i} \l(u)\dd u}
\le 
\frac{\l(t_{i-2})}{ (t_i-t_{i-2})\l(t_{i})}
\le \frac{\l(t_{i-2})}{\beta n^{\frac{\a}{1-\a}}\l(t_{i})}
= \frac1{\beta n^{\frac{\a}{1-\a}}}
\cdot\left(\frac{t_i}{t_{i-2}}\right)^\a
\\ &
\le 
\frac1{\beta n^{\frac{\a}{1-\a}}}
\cdot\left(\frac{\tau_n}{\tau_{n/2}}\right)^\a  
\le 
\frac1{\beta n^{\frac{\a}{1-\a}}}
\cdot\left(\frac{c_1}{c_0 2^{-\frac{1}{1-\a}}}\right)^\a  
\end{align*}
since the events $\mathcal{E}_1$ and $\mathcal{E}_2$ occur. This implies the stated result with $c=\b^{-1}\left(c_0^{-1} c_1 \sqrt[1-\a]{2} \right)^\a $.
\end{proof}

Now we divide $W_n$ into two portions:
\begin{align*}
   A&=\sum_{i\in J_n} \xi_i; \\
   B&=W_n-A=\sum_{i\in \{2,4,\dots,n\}\setminus  J_n} \xi_i; 
\end{align*}

\begin{lemma}\label{lemm1K}
We have
$$
\sup_{x\in\R} \P \left(
A\in[x-\rho,x+\rho]\mid \mathcal{D}
\right)
\leq\ \frac{C}{n^{\frac{\alpha}{1-\alpha}+\frac12}}
$$ 
for some $C=C(\a,\rho)>0$.
\end{lemma}
\begin{proof}[Proof of Lemma~\ref{lemm1K}]
The result follows immediately from Lemma~\ref{lem:kesten} with $a_i\equiv 2\rho=L$, using Lemma~\ref{lemm1}, and the fact that $\card{J_n}\ge n/17$.
\end{proof}

So,
\begin{align*}
\P(|W_n|\le \rho\mid \mathcal{D})
&=\P(A+B\in[-\rho,\rho]\mid \mathcal{D})
=\int \P(A+b\in[-\rho,\rho]\mid \mathcal{D})
f_{B|\mathcal{D}}(b) \dd b    
\\
&\le 
\int \sup_x \P(A\in[x-\rho,x+\rho]\mid \mathcal{D})
f_{B|\mathcal{D}}(b) \dd b
\le \frac{C}{n^{\frac{\alpha}{1-\alpha}+\frac12}}
\end{align*}
where $f_{B|\mathcal{D}}(\cdot)$ is the density of $B$ conditional on $\mathcal{D}$.

Finally, using~\eqref{eq:Ac} and~\eqref{eqD1} we get
\begin{align*}
    \P(W_n\in [-\rho,\rho])\le
    \P(W_n\in [-\rho,\rho]\mid \mathcal{D})
    +\P\left(\mathcal{D}_1^c\right)
    +\P\left(\mathcal{E}_2^c\right)
    \le 
    \frac{C}{n^{\frac{\alpha}{1-\alpha}+\frac12}}+
    5 e^{-c_* n}+e^{-n/15}
\end{align*}
which is summable over $n$, so we can apply the Borel-Cantelli lemma to show that $\{|W_n|\le \rho\}$ occurs finitely often a.s.
\end{proof}

\begin{thm} \label{thm2}
Let $d\ge 2$,  $\alpha\in(0,1)$, and the rate of the PPP is given by~\eqref{rate1}. Then $Z(t)$ is transient a.s.
\end{thm}
\begin{rema}
The above result holds also for $\a=1$, and the proof is more or less identical to that of Theorem~5.2 in~\cite{EV}, once we establish that a.s.\ $\tau_n>e^{cn}$ for some $c>0$ and all large $n$; the latter follows from the arguments similar to that of Lemma~\ref{lemlarg}.
\end{rema}
\begin{proof}
We will provide the proof only for the case $d=2$ and $\rho=1$; it can be easily generalized for all $d\ge 3$ and $\rho>0$. Denote the coordinates of the embedded walk by $X_n$ and $Y_n$, thus $W_n=(X_n,Y_n)\in\R^2$.  Fix some small $\eps>0$ and consider the event
\begin{align*}
\mathcal{R}_n&=\left\{Z(t)\in [-1,1]^2\text{ for some }t\in (\tau_n,\tau_{n+1}]\right\}
\end{align*}
We will show that events $\mathcal{R}_n$ occur finitely often a.s., thus ensuring the transience of $Z(t)$.

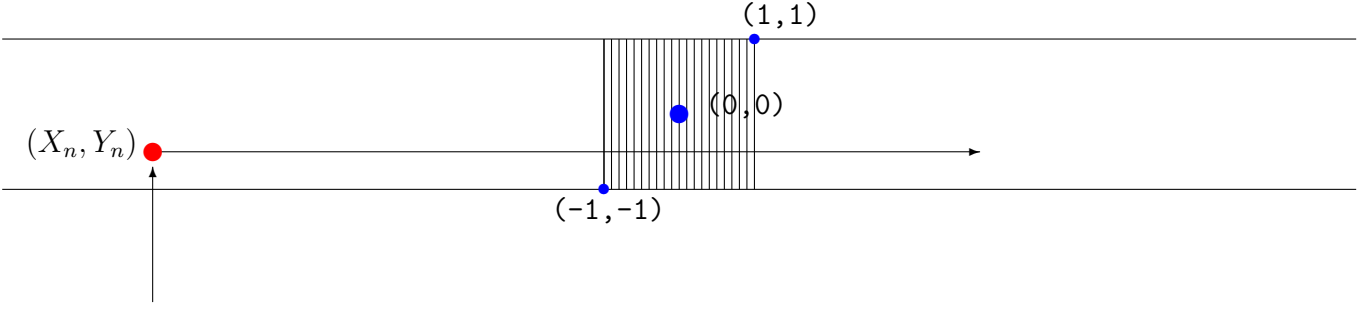
\begin{figure}
\setlength{\unitlength}{1cm}
\begin{picture}(18,6)
\put(0,2){\line(1,0){18}}
\put(0,4){\line(1,0){18}}
\put(8,2){\line(0,1){2}}
\put(10,2){\line(0,1){2}}
\multiput(8,2)(0.1,0){20}{\line(0,1){2}}
\put(2,0.5){\vector(0,1){1.8}}
\put(2,2.5){\vector(1,0){11}}
\put(9,3){{\color{blue}\circle*{0.25}}\hbox{\kern3pt \texttt{(0,0)}}}
\put(2,2.5){{\color{red}\circle*{0.25}}\hbox{\kern-55pt \texttt{$(X_n,Y_n)$}}}
\put(9.8,4.2){\texttt{(1,1)}}
\put(7.3,1.6){\texttt{(-1,-1)}}
\put(8,2){\color{blue}\circle*{0.125}}
\put(10,4){\color{blue}\circle*{0.125}}
\end{picture}    
    \caption{Model A: recurrence of conservative random walk on $\R^2$}
    \label{fig:1}
\end{figure}

For the walk $Z(t)$ to hit $[-\rho,\rho]^2$ between times $\tau_n$ and $\tau_{n+1}$, we have to have either 
\begin{itemize}
    \item $|X_n|\le \rho$, ${\bf f}_n=-\mathrm{sign}(Y_n) {\bf e}_2$, and $ \tau_{n+1}-\tau_n\ge |Y_n|-1$, or
    \item $|Y_n|\le \rho$, ${\bf f}_n=-\mathrm{sign}(X_n) {\bf e}_1$, and $ \tau_{n+1}-\tau_n\ge |X_n|-1$
\end{itemize}
(see Figure~\ref{fig:1}.)

Using the fact that ${\bf f}_n$ are independent of anything, we get
$$
\P(\mathcal{R}_n)\le 
\frac14\, \P\left(|X_n|\le \rho, \tau_{n+1}-\tau_n\ge |Y_n|-1\right) 
+
\frac14\, \P\left(|Y_n|\le \rho, \tau_{n+1}-\tau_n\ge |X_n|-1\right).
$$
We will show that $\P\left(|X_n|\le \rho,|Y_n|\le \tau_{n+1}-\tau_n+1\right)$ is summable in $n$ (and the same holds for the other summand by symmetry), hence by the Borel-Cantelli lemma that $\mathcal{R}_n$ occur finitely often a.s.

\begin{lemma}\label{lem:Y}
Assume that $\eps\in (0,1)$. Then  there exists $c_1^*=c_1^*(\eps)$ such that for all sufficiently large $n$
    $$
    \P\left(\tau_{n+1}-\tau_n+1\ge n^{\frac{\a}{1-\a}+\eps}\right)
    \le 3\, e^{-c_1^* n^\eps}
    $$
\end{lemma}
\begin{proof}[Proof of Lemma~\ref{lem:Y}]
For $s,t\ge 0$ we have
$$
  \P(\tau_{n+1}-\tau_n\ge s\mid \tau_n=t)
  =\exp\left\{-[\Lambda(t+s)-\Lambda(t)]\right\}  
  =\exp\left\{-\frac{(t+s)^{1-\a}-t^{1-\a}}{1-\a}\right\},
$$
so if $t<c_1 n^{\frac1{1-\a}}$, where $c_1$ is the constant from Lemma~\ref{lemlarg}, and $s=n^{\frac{\a}{1-\a}+\eps}-1=o\left(n^{\frac1{1-\a}}\right)$, we have
\begin{align*}
  \P(\tau_{n+1}-\tau_n\ge s\mid \tau_n=t) &\le 
  \exp\left\{-\frac{(c_1 n^{\frac1{1-\a}}+s)^{1-\a}-c_1^{1-\a} n}{1-\a}\right\}  
  \\ &\le
  \exp\left\{-\frac{n c_1^{1-\a}}{1-\a}\left[\left(1+\frac{n^{\frac{\a}{1-\a}+\eps}-1}{c_1 n^{\frac1{1-\a}}}\right)^{1-\a}-1\right] \right\}  
  =\exp\left\{-\frac{n^\eps (1+o(1))}{c_1^{\a}}\right\}.
\end{align*}
 Consequently, using  Lemma~\ref{lemlarg},
\begin{align}\label{eqnotlarge}
    \P\left(\tau_{n+1}-\tau_n\ge n^{\frac{\a}{1-\a}+\eps}-1\right)\nonumber 
    &\le 
    \P\left(\tau_{n+1}-\tau_n\ge n^{\frac{\a}{1-\a}+\eps}-1 \mid \tau_n< c_1 n^{\frac1{1-\a}}\right)
    +\P\left(\tau_n\ge c_1 n^{\frac1{1-\a}}\right) 
    \\ & \le e^{-c_1^* n^\eps}+e^{-n/15}
    \le 2 e^{-c_1^* n^\eps}
\end{align}
for some $c_1^*\in(0,1/15)$.
\end{proof}
Now we will modify the proof of Theorem~\ref{th:A1} slightly to adapt to our needs.  Let the events~$\mathcal{E}_1$ and~$\mathcal{E}_2$ be the same as in the proof of this theorem. We will also use the set $I_n$, but instead of $J_n$ we will  introduce the sets
\begin{align*}
    J_n^1 &=\{i\in I_n:\ {\bf f}_{i-1}=-{\bf f}_{i-2}, \ {\bf f}_{i-1}\in\{ {\bf e}_1,-{\bf e}_1\}\};\\
    J_n^2 &=\{i\in I_n:\ {\bf f}_{i-1}=-{\bf f}_{i-2}, \ {\bf f}_{i-1}\in\{ {\bf e}_2,-{\bf e}_2\}\}.
\end{align*}
Similarly to the proof of Theorem~\ref{th:A1}, inequality~\eqref{eqD1}, we immediately obtain that
\begin{align}\label{eq:JA2}
\begin{split} 
\P(\card{J_n^1}\le n/34)&\le
\P(\card{J_n^1}\le n/34\mid \mathcal{E}_1)+\P(\mathcal{E}_1^c)\le 6 e^{-c_*'n},\\
\P(\card{J_n^2}\le n/34)&\le
\P(\card{J_n^2}\le n/34\mid \mathcal{E}_1)+\P(\mathcal{E}_1^c)\le 6 e^{-c_*'n} 
\end{split}
\end{align}
for some $c_*'>0$.

Fix a deterministic sequence of unit vectors ${\bf g}_1,\dots,{\bf g}_n$ such that each ${\bf g}_i\in\{\pm{\bf e}_1,\pm{\bf e}_2\}$.
We now also define the event
\begin{align*}
\mathcal{D}&=
\{{\bf f}_1={\bf g}_1,\dots,{\bf f}_n={\bf g}_n\}
\cap
\{\tau_k=t_k\text{ for all } k\le n:\ k\text{  is even or }
{\bf f}_k\perp {\bf f}_{k+1}\}
\\ &
\cap \left\{\card{J^1_n}\ge \frac{n}{34}\right\}
\cap \left\{\card{J^2_n}\ge \frac{n}{34}\right\}\cap \mathcal{E}_1 \cap \mathcal{E}_2.
\end{align*}
Therefore, repeating the previous arguments for each of the horizontal  and vertical components of $W_n$ we immediately obtain
\begin{align*}
\P(|X_n|\le 1\mid \mathcal{D})
&\le \frac{C}{n^{\frac{\alpha}{1-\alpha}+\frac12}};\\
\sup_{x\in\R}\P(|Y_n-x|\le 1\mid \mathcal{D})
&\le \frac{C}{n^{\frac{\alpha}{1-\alpha}+\frac12}}.
\end{align*}
The second inequality implies that
\begin{align*}
\P(|Y_n|\le n^{\frac{\a}{1-\a}+\eps} \mid \mathcal{D})
\le  n^{\frac{\a}{1-\a}+\eps} \cdot \frac{C}{ n^{\frac{\alpha}{1-\alpha}+\frac12}}
=\frac{C}{ n^{\frac12-\eps}}.
\end{align*}
Now, by Lemma~\ref{lem:Y}
\begin{align*}
\P\left(|X_n|\le 1,|Y_n|\le \tau_{n+1}-\tau_n+1\right)
&\le \P\left(|X_n|\le 1,|Y_n|\le \tau_{n+1}-\tau_n+1\mid \tau_{n+1}-\tau_n+1\le n^{\frac{\a}{1-\a}+\eps}\right)\\ & +
\P\left(\tau_{n+1}-\tau_n+1\ge n^{\frac{\a}{1-\a}+\eps}\right)
\\ & \le
\P\left(|X_n|\le 1,|Y_n|\le  n^{\frac{\a}{1-\a}+\eps}\right)  +
3 e^{-c_1^*n^\eps}.
\end{align*}
Observing that $X_n$ and $Y_n$ are actually independent given $\mathcal{D}$, we conclude
that
\begin{align*}
\P\left(|X_n|\le 1,|Y_n|\le  n^{\frac{\a}{1-\a}+\eps}\right) &\le 
\P\left(|X_n|\le 1,|Y_n|\le  n^{\frac{\a}{1-\a}+\eps}\mid \mathcal{D}\right)\\ & +\P\left(\card{J_n^1}<\frac{n}{34}\text{ or }\card{J_n^1}<\frac{n}{34}\right)
+\P(\mathcal{E}_1^c)+\P(\mathcal{E}_2^c) 
\\ &
\le 
\frac{C}{n^{\frac{\a}{1-\a}+\frac12}}
\cdot
\frac{C}{n^{\frac12-\eps}}
+12\, e^{-c'_* n}+ e^{-n/15}+ e^{-n/15}
=\frac{C^2+o(1)}{n^{1+\left[\frac{\a}{1-\a}-\eps\right]}}
\end{align*}
by~\eqref{eq:Ac} and~\eqref{eq:JA2}. Assuming $\eps\in\left(0,\frac{\a}{1-\a}\right)$, the RHS is summable, and thus $\P\left(|X_n|\le \rho,|Y_n|-1\le \tau_{n+1}-\tau_n\right)$ is also summable in $n$.
\end{proof}
\vskip 1cm

\begin{thm}\label{th:sat} 
Let $d\geq2$, and  the rate  be given by
\begin{align}\label{rate2}
\l(t)=
\begin{cases}
\frac{1}{(\ln{t})^\beta}, &t\ge e;\\
0, &\text{otherwise}.
\end{cases}   
\end{align}
Then $Z(t)$ is transient as long as $\beta >2.$
\end{thm}
\begin{proof}
The proof is analogous to the proof of Theorem~\ref{thm2}; we will only indicate how that prove should be modified to this case. As before, we assume that $\rho=1$, $d=2$, without loss of generality. First, we prove the following
\begin{lemma}\label{lem:ln}
Let 
$$
\Lambda(T)=\int_0^T\l(s)\dd s=\int_e^T\frac{\dd s}{(\ln s)^\beta}.
$$
Then
$$
\lim_{T\to\infty}\frac{ \Lambda(T)}{T(\ln T)^{-\beta}}=1
$$
\end{lemma}
\begin{proof}[Proof of Lemma~\ref{lem:ln}]
Fix an $\epsilon\in(0,1)$, then
$$
\Lambda(T)=\int_e^T\frac{\dd s}{(\ln s)^\beta}=\int_e^{T^{1-\epsilon}}\frac{\dd s}{(\ln s)^\beta}
+ \int_{T^{1-\epsilon}}^T\frac{\dd s}{(\ln s)^\beta}
<\ T^{1-\epsilon}+\frac{T}{(\ln T)^\beta (1-\epsilon)^\beta}.
$$    
At the same time, trivially,
$$
\Lambda(T)>\frac{T-e}{(\ln T)^\beta}.
$$
Now the limit of the ratio of the upper and the lower bounds of ${\Lambda(T)}$ can be made arbitrarily close to~$1$, by choosing a small enough $\epsilon$. Hence the statement of the lemma follows.
\end{proof}

The rest of the proof goes along the same lines as that of Theorem~\ref{thm2} 
\begin{itemize}
\item[(1)]
First note that for some  $c>0$ the event  $\left\{\tau_n\ \leq\ cn(\ln{n})^{\beta}\right\}$ occurs finitely often almost surely. Indeed, 
$$
\tilde\Lambda=\Lambda(cn(\ln{n})^{\beta})=
\frac{cn(\ln{n})^{\beta}}{\left[\ln(cn(\ln{n})^{\beta})\right]^{\beta}}(1+o(1))=\frac{(\ln n)^\beta}{(1+o(1))(\ln n)^\beta}cn(1+o(1))=(1+o(1))cn
$$
by Lemma~\ref{lem:ln}, and thus
\begin{align*}
\P(\tau_n\leq\ cn(\ln{n})^{\beta})&=
\P\left(N\left(cn(\ln{n})^{\beta}\right) \ge n\right)= e^{-\tilde\Lambda} \left(    \frac{\tilde\Lambda^n}{n!}+\frac{\tilde\Lambda^{n+1}}{(n+1)!}+\frac{\tilde\Lambda^{n+2}}{(n+2)!}+\dots\right)
    \\ &
= e^{-\tilde\Lambda}\frac{\tilde\Lambda^n}{n!}
\left(1+\frac{\tilde\Lambda}{(n+1)}+\frac{\tilde\Lambda^2}{(n+1)(n+2)}+\dots\right)
\\  &
\le e^{-\tilde\Lambda}\frac{\tilde\Lambda^n}{n!}
\left(1+\frac{\tilde\Lambda}{1!}+\frac{\tilde\Lambda^2}{2!}+\dots\right)
=\frac{\tilde\Lambda^n}{n!}= \frac{\left[(1+o(1))cn\right]^{n}}{n!}
=\mathcal{O}\left(\frac{[(1+o(1)) c]^n n^{n}  }{n^n e^{-n} \sqrt{n} }\right). 
\end{align*}
This quantity is summable as long as $ce<1$, and the statement follows from the Borel-Cantelli lemma.
    
\item[(2)] For any positive $\epsilon$, the event $\bigl\{\tau_n\ \geq\ c^*n(\ln{n})^{\beta}\bigl\}$, where $c^*=1+\epsilon$, occurs finitely often, almost surely. Indeed, 
$$
\bar\Lambda:=\Lambda\left( c^*n(\ln{n})^{\beta}\right)
=\frac{c^*n(\ln{n})^{\beta}}{[\ln(c^*n(\ln{n})^{\beta})]^\beta} (1+o(1))=c^* n (1+o(1))\ge (1+\epsilon/2)n
$$
for large enough $n$ by Lemma~\ref{lem:ln}. Hence, since $\bar\Lambda>n$,
\begin{align*}
   \P\bigl(\tau_n\ \geq\ c^*n(\ln{n})^{\beta}\bigl)&=\P(N\left(c^*n(\ln{n})^\beta\right)\le n)
   = e^{-\bar\Lambda} \Bigl(1+\bar\Lambda+\frac{{\bar\Lambda}^2}{2!}+\dots+\frac{{\bar\Lambda}^n}{n!}\Bigl)\\ &
   \le 
   e^{-\bar\Lambda}\, (n+1)\,\frac{{\bar\Lambda}^n}{n!}
   =e^{-\bar\Lambda}\, (n+1)\,\frac{{\bar\Lambda}^n}{n^n e^{-n}\sqrt{2\pi n}}(1+o(1))
   \\&
=(1+o(1))\sqrt{\frac{n}{2\pi}}\exp\left\{
-n\left[\frac{\bar\Lambda}{n}-1-\ln \frac{\bar\Lambda}{n}\right]
\right\}
\end{align*}
which is summable in $n$, as $c^*=1+\epsilon$ implies that the expression in square brackets is strictly positive (this follows from an easy fact that $c-1-\ln c>0 $ for $c>1$).
Hence, $\left\{\tau_n\ \geq\ c^*n(\ln{n})^{\beta}\right\}$ happens finitely often, almost surely, as stated.

\item[(3)] The event $\left\{\tau_{n+1}-\tau_{n}\ \geq\ 2 (\ln{n})^{1+\beta}\right\}$ occurs finitely often, almost surely.
This holds because, for all sufficiently large $n$, we have $\tau_{n+1}\le c^* (n+1)(\ln (n+1))^\b$, hence for $t\le \tau_{n+1}$
$$
\l(t)\ge \frac{1}{[\ln(c(n+1)(\ln (n+1))^\b)]^\b}=\frac{1+o(1)}{(\ln n)^\b}
$$
and thus
$(\tau_{n+1}-\tau_{n})$ is stochastically smaller than an exponential random variable $\mathcal{E}$ with parameter $\displaystyle\frac{1+o(1)}{(\ln{n})^\beta}$. So for all sufficiently large $n$,
\begin{align}\label{eqnotlarge2}
    \P(\tau_{n+1}-\tau_n\geq\ 2(\ln{n})^{1+\beta})\leq\ \P(\mathcal{E}\geq\ 2(\ln{n})^{1+\beta})=\frac{1}{n^{2-o(1)}}
\end{align}
which is summable in $n$, and we can apply the Borel-Cantelli lemma.

Using previous arguments for horizontal and vertical components, we obtain,
\begin{align*}
\P\bigl(|X_n|\leq 1\bigl)&\leq \frac{C}{\sqrt{n}(\ln{n})^{\beta}}\\
\P\bigl(|Y_n|\leq (\ln{n})^{1+\beta}\bigl)&\leq(\ln{n})^{1+\beta}\ \times\ \frac{C}{\sqrt{n}(\ln{n})^{\beta}}\ =\ \frac{C\ln{n}}{\sqrt{n}}
\end{align*}
Again, similar to Theorem~\ref{thm2}, the event $\bigl\{|X_n|\leq \rho,\ |Y_n|\leq 2(\ln{n})^{1+\beta}\bigl\}$ has to happen infinitely  often almost surely, for $Z(t)$ to be recurrent.

Thus it will follow, that 
\begin{align*}
        \P(Z(t)\text{ visits }[-1,1]^2\text{ for }t\in[\tau_n,\tau_{n+1}])&\leq\P\bigl(|X_n|\leq 1,\ |Y_n|\leq 2(\ln{n})^{1+\beta}\bigl)+f(n)\\&\leq\ \frac{C^2}{n(\ln{n})^{\beta-1}} +f(n)+ g(n),
\end{align*}
where $f$ and $g$ are two summable functions over $n$, similar to Theorem~\ref{thm2}. The RHS is summable over $n$ as long as $\b>2$. Hence $Z(t)$ visits $[-1, 1]^2$ finitely often, almost surely.
\end{itemize}
\end{proof}

\section{Analysis of Model B}
Recall that in Model B, vectors ${\bf f}_n$ are uniformly distributed over the unit sphere in $\R^d$.

Throughout this section, we will again suppose that $\rho=1$ and we will show that the process is not $1-$recurrent. The proof for the general $\rho$ is analogous and is omitted.

Let   $W_n=(X_n,Y_n,*,*,\dots,*)$ be the embedded version of the process $Z(t)=(X(t),Y(t),*,*,\dots,*)$; here $X_n$ and $Y_n$ ($X(t)$ and $Y(t)$ respectively) stand for the process' first two coordinates. We denote the projection of $Z(t)$ ($W_n$ respectively) on the two-dimensional plane by $\hat W(t)=(X(t),Y(t))$ ($\hat W_n=(X_n,Y_n)$ respectively.) See Figure~\ref{fig:CRW}.

\begin{figure}
    \centering
    \includegraphics[scale=0.2]{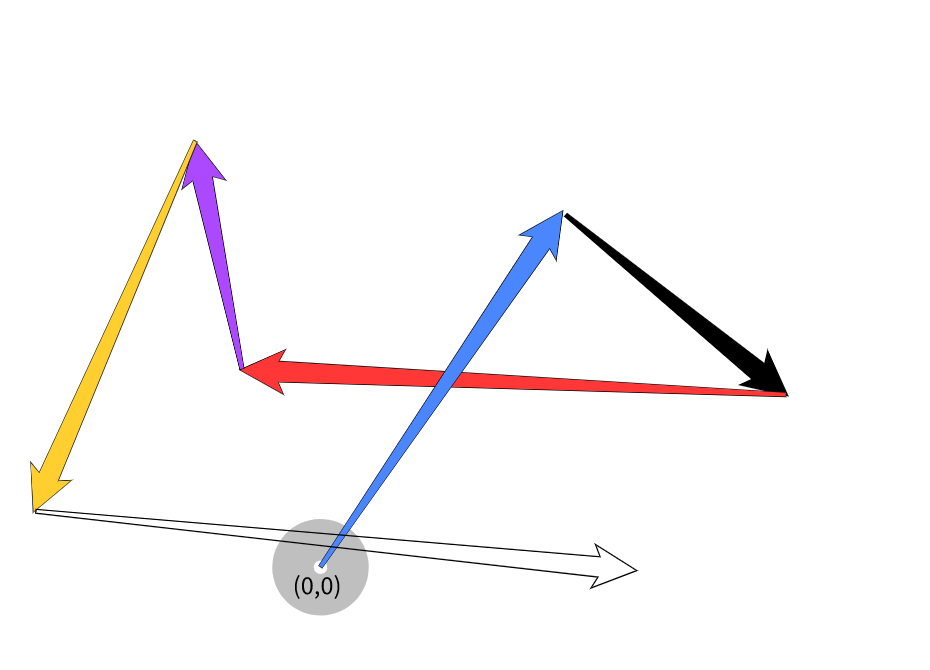}
    \caption{Model B: $\hat W_n$, (the projection of) conservative random walk on $\R^2$}
    \label{fig:CRW}
\end{figure}

The following statement is quite intuitive.
\begin{lemma}\label{lemSat}
Suppose $d\ge 3$, and let ${\bf f}=(f_1,f_2,\dots,f_d)$ be a random vector uniformly distributed on the unit sphere $\mathcal{S}^{d-1}$ in $\R^d$. Then for some $\g>0$ we have
$$
\P\left(\sqrt{f_1^2+f_2^2}\ge \g\right)\ge \frac23.
$$
\end{lemma}
\begin{rema}
    The statement is trivially true for the case $d=2$ as well.
\end{rema}
\begin{proof}
We will use the following well-known representation (see e.g.~\cite{Blum}, Section 2.5) of ${\bf f}$:
$$
{\bf f}=\left(\frac{\eta_1}{\Vert\eta\Vert},\dots,\frac{\eta_d}{\Vert\eta\Vert}\right)
$$
where $\eta_i$, $i=1,2,\dots$, are i.i.d.\ standard normal and $\Vert\eta\Vert=\sqrt{\eta_1^2+\dots+\eta_d^2}$.
For some large enough $A>0$ we have
$$
\P\left(\max_{i=3,\dots,d} |\eta_i|<A\right)=\left(\Phi(A)-\Phi(-A)\right)^{d-2}\ge \sqrt{\frac 23},
$$
where $\Phi(\cdot)$ is the distribution function of the standard normal random variable. Also,  for some small enough $a\in(0,A)$
$$
\P\left(\max_{i=1,2} |\eta_i|>a\right)\ge \sqrt{\frac 23}.
$$
On the intersection of these two independent events we have
$$
f_1^2+f_2^2=\frac{\eta_1^2+\eta_2^2}{(\eta_1^2+\eta_2^2)+(\eta_3^2+\dots+\eta_d^2)}\ge \frac{a^2}{a^2+(d-2)A^2}=:\g^2.
$$
\end{proof}
\begin{rema}
In fact, one can rigorously compute
\begin{align*}
\P\left(f_1^2+f_2^2\ge \g^2\right)&=
\P\left(\frac{\eta_1^2+\eta_2^2}{\eta_1^2+\eta_2^2+\dots+\eta_d^2}\ge \g^2\right)
=
\P\left(\eta_1^2+\eta_2^2
\ge \frac{\g^2}{1-\g^2}(\eta_3^2++\dots+\eta_d^2)\right)
\\ &
=\P\left(\chi^2(2)
\ge \frac{\g^2}{1-\g^2}\chi^2(d-2)\right)
=\iint_{x\ge \frac{\g^2}{1-\g^2}y\ge 0} \frac{e^{-x/2}}2  \times \frac{y^{d/2-2}e^{-y/2}}{2^{d/2-1} \Gamma(d/2-1)} \dd x \dd y \\ & 
=\int_0^\infty  
e^{-\frac{y}2 \cdot\frac{\g^2 }{2(1-\g^2)}}  \times \frac{y^{d/2-2}e^{-y/2}}{2^{d/2-1} \Gamma(d/2-1)} \dd y
=(1-\g^2)^{d/2-1},
\end{align*}
however, we do not really need this exact expression.
\end{rema}

\begin{lemma}\label{lemCV1}
Suppose $d\ge 2$, and let $R_k=\{(x,y)\in\R^2: k^2\le x^2+y^2\le (k+1)^2\}$ be the ring of radius $k$ and width $1$ centred at the origin.  For some constant $C>0$, possibly depending on $d$ and $\a$,
\begin{align*}
\P\left(\hat W_n^{(1)}\in R_k\right)&\le \frac{Ck}{n^{1+\frac{2\a}{1-\a}}};
\\
\P\left(\hat W_n^{(2)}\in R_k\right)&\le \frac{Ck}{n(\ln{n})^{2\beta}}
\end{align*}
for all large $n$, where $\hat W_t^{(1)}$ is the walk with rate~\eqref{rate1}  and $\hat W_t^{(2)}$ is the walk with rate with rate~\eqref{rate2} respectively.
\end{lemma}
\begin{proof}
Assume that the event $\mathcal{E}_1$ defined by~\eqref{eq:EE} has occurred.
We can write
$$
\hat W_n=(X_n,Y_n)=\sum_{k=1}^n (\tau_k-\tau_{k-1})\tilde{\mathbf{f}}_k
\quad\text{where}\quad
\tilde{\mathbf{f}}_k=\ell_k\left[\mathbf{e_1}\cos(\phi_k)+\mathbf{e_2}\sin(\phi_k)\right]
$$
and $\phi_k$, $k=1,2,\dots$, are  uniformly distributed on $[-\pi,\pi]$, and $\ell_k$ is a length of projection $\tilde{\mathbf{f}}_k$ of $\mathbf{f}_k$ on the two-dimensional plane. Note that elements of the set $\{\ell_1,\eta_2,\eta_3,\dots,\phi_1,\phi_2,\phi_3,\dots\}$ are all independent. Also, define 
$$
\mathcal{D}_2=\{\text{for at least half of the integers   }i\in[n/2,n]
\text{ we have }\ell_i\ge \g\}
$$
where $\g$ is the constant from Lemma~\ref{lemSat}.
Then $\P(\mathcal{D}_2^c)\le 2e^{-n/36} $ by~\eqref{LDP} and Lemma~\ref{lemSat}.

Let $\tilde W_n$ be the distribution of $\hat W_n\in\R^2$ {\em conditioned} on $\mathcal{E}_1\cap  \mathcal{D}_2$,  and
$$
\varphi_{\tilde W_n}(t)=\E e^{i t\cdot \tilde W_n}=\E \exp\left\{i\sum_{k=1}^n t\cdot \tilde{\mathbf{f}}_{k}(\tau_k-\tau_{k-1})\right\}
$$
be its characteristic function (here $t\cdot \tilde{\mathbf{f}}_k=\ell_k\left(t_1\cos(\phi_k)+t_2\sin(\phi_k)\right)$).
We will use the L\'evy inversion formula, which allows us to compute the density of $\tilde W_n$, provided $|\varphi_{\tilde W_n}(t)|$ is integrable:
\begin{align}\label{eqFT}
f_{\tilde W_n}(x,y)&=\frac1{(2\pi)^2} \iint_{\R^2}
e^{-i( t_1 x+t_2 y)} \varphi_{\tilde W_n}(t)
\dd t_1 \dd t_2 
\le 
\frac1{(2\pi)^2} \iint_{\R^2}
\left|\varphi_{\tilde W_n}(t)\right|
\dd t_1 \dd t_2 .
\end{align}

Let $\Delta_k=\tau_k-\tau_{k-1}$. Since $\phi_1,\dots,\phi_n$ are i.i.d.\ Uniform$[-\pi,\pi]$ and independent of anything, we have
\begin{align*}
\varphi_{\tilde W_n}(t)
 &=   \E \left[ \prod_{k=1}^n \E \left(e^{i\Delta_k\ell_k(t_1\cos\phi_k+t_2\sin\phi_k)} 
\mid \tau_1,\dots,\tau_n;\ell_1,\dots,\ell_n\right)\right]
= \E \left[ \prod_{k=1}^n J_0\left(\Vert t\Vert \Delta_k\ell_k \right)\right]
\\
\Longrightarrow
\\
|\varphi_{\tilde W_n}(t)|&\le 
\E \left[ \prod_{k=1}^n \left|J_0(\Vert t\Vert \Delta_k\ell_k )\right|\right]
\end{align*}
where $\Vert t\Vert=\sqrt{t_1^2+t_2^2}$, and $J_0(x)=\sum_{m=0}^\infty \frac{(-x^2/4)^{m}}{m!^2}$ is Bessel $J_0$ function. Indeed, for any $x\in\R$, setting $\tilde x=x\sqrt{t_1^2+t_2^2}$ and $\b=\arctan \left(\frac{t_2}{t_1}\right)$, we get
\begin{align*}
\E \left[e^{i x (t_1\cos\phi_k+t_2\sin\phi_k)}\right]
&=\frac1{2\pi} \int_0^{2\pi}e^{ix(t_1\cos\phi+t_2\sin\phi)}\dd \phi
=\frac1{2\pi} 
\int_0^{2\pi}e^{i\tilde x(\cos\beta\cos\phi+\sin\beta \sin\phi)}\dd \phi
\\ & = \frac1{2\pi} \int_0^{2\pi}e^{i\tilde x \cos(\phi+\beta)}\dd \phi
 =
\frac1{2\pi} \int_0^{2\pi}e^{i\tilde x \cos(\phi)}\dd \phi
=\frac1{2\pi} \int_0^{2\pi} \cos(\tilde x\cos(\phi))\dd\phi=J_0(\tilde x)
\end{align*}
due to periodicity of $\cos(\cdot)$ and the fact that the function $\sin(\cdot)$ is odd, and formula (9.1.18) from~\cite{AS}.

Let $\xi_k$ be the random variable with the distribution of $\Delta_k$ given $\tau_{k-1}$. Then, recalling that we are on the event $\mathcal{E}_1$, we get
\begin{align*}
\xi_k\succ     \tilde\xi^{(1)} &\text{ for }\hat{W}^{(1)}_n;\\
\xi_k\succ     \tilde\xi^{(2)} &\text{ for }\hat{W}^{(2)}_n
\end{align*}
where
$\tilde\xi^{(1)}$ ($\tilde\xi^{(2)}$ respectively) is an exponential random variable with rate $1/{a_1}$ ($1/a_2$ respectively) with
\begin{align*}
a_1&=\tilde c_1 n^{\frac{\a}{1-\a}}    ,\\
a_2&=\tilde c_2 (\ln{n})^\beta
\end{align*}
for some constants $\tilde c_1>0$ and $\tilde c_2>0$, and ``$\zeta_a\succ \zeta_b$'' denotes that random variable $\zeta_a$ is stochastically larger than random variable $\zeta_b$.
Consequently, setting $a=a_1$, $\tilde\xi=\tilde\xi^{(1)}$ or $a=a_2$, $\tilde\xi=\tilde\xi^{(2)}$ depending on which of the two models we are talking about, and $\mathcal{F}_k=\sigma(\tau_1,\dots,\tau_k)$, we get
\begin{align}\label{eq:h}
&\E\left[ \left|J_0(\Vert t\Vert \Delta_k\ell_k)\right| \mid \F_{k-1},\ell_k=\ell \right]=\E\,|J_0(\Vert t\Vert \xi_k \ell)|
\le   \E\, G(\Vert t\Vert \xi_k \ell)
\le 
\E\, G(\Vert t\Vert \tilde\xi\ell)\\ \nonumber &
=\int_0^\infty
a^{-1} e^{-{y}/{a}} G(\Vert t\Vert y \ell) \dd y
=\int_0^\infty  e^{-u} G(s u) \dd u
=\int_0^\infty  \frac{e^{-u}}{\sqrt[4]{1+s^2 u^2}} \dd u
=:h(s,\ell)
\end{align}
where $$s=a \ell \Vert t\Vert ,$$ since $|J_0(x)|\le G(x)$ by~\eqref{Besbound2}  and the fact that $G(\cdot)$ is a decreasing function and $\xi_k\succ \tilde\xi$.

We will now estimate the function $h(t,\ell)$. 
Since $|J_0(x)|\le 1$, we trivially get $0\le h(s,\ell)\le 1$.
Additionally, for all $s>0$
\begin{align*}
 h(s,\ell)
\le  \int_{0}^\infty  \frac{e^{-u}}{\sqrt{su}} \dd u
= \sqrt{\frac{\pi}{s}}= \sqrt{\frac{\pi}{a \ell \Vert t\Vert}}. 
\end{align*}
Let 
$$
n/2\le j_1<j_2<\dots<j_m\le n
$$
be the indices $i\in[n/2,n]$ for which $\ell_i\ge\g$. 
Since $|J_0(x)|\le 1$, we have from~\eqref{eq:h}
\begin{align*}
\E \left[ \prod_{k=1}^n \left|J_0(\Vert t\Vert \Delta_k\ell_k) \right|\right]
& \le 
\E \left[ \prod_{i=1}^m \left|J_0(\Vert t\Vert \Delta_{j_i}\ell_{j_i}) \right| \right]
=\E \left[ \E\left(
|J_0(\Vert t\Vert \Delta_{j_m}\ell_{j_m})|
\mid \F_{j_m-1},\ell_{j_m}\right)
\prod_{i=1}^{m-1}  \left|J_0(\Vert t\Vert \Delta_{j_i}\ell_{j_i}) \right| \right]
\\
&\le \sqrt{\frac{\pi}{a \g \Vert t\Vert }}\cdot
\E \left[\prod_{i=1}^{m-1}  \left|J_0\left(\Vert t\Vert \Delta_{j_i}\ell_{j_i}\right) \right| \right]
\end{align*}
By iterating this argument for $i=j_{m-1},j_{m-2},\dots,j_1$, we get
$$
|\varphi_{\tilde W_n}(t)|\le \E \left[ \prod_{k=1}^n \left|J_0(\Vert t\Vert \Delta_k\ell_k) \right|\right]
\le \left(\sqrt{\frac{\pi}{a \g \Vert t\Vert }}\right)^{n/4}
$$
(recall that $m\ge n/4$ on $\mathcal{D}_1$).

Consider now two cases.  For $\Vert t\Vert\ge \frac{2\pi}{a\g}$ a part of the inversion formula gives
\begin{align}\label{bound1}
\iint_{\Vert t\Vert>\frac{2\pi}{a\g}}
\left|\varphi_{\tilde W_n}(t) \right|\dd t_1 \dd t_2 
& \le 
\iint_{\Vert t\Vert\ge \frac{2\pi}{a\g}}
\left(\frac{\pi}{\g \Vert t\Vert a }\right)^{n/8}
\dd t_1 \dd t_2 
=
\left(\frac{\pi}{a\g }\right)^2\int_0^{2\pi}\dd\theta
\int_2^\infty 
\frac{r \dd r}{r^{n/8}}
\\ & =
\left.\frac{2\pi^3}{a^2\g^2}\times \frac{r^{2-n/8}}{2-n/8}\right|_2^\infty = o\left(2^{-n/8}\right)
\nonumber
\end{align}
by changing the variables $t_1=\frac{\pi}{a\g} r \cos\theta$,
$t_2=\frac{\pi}{a\g} r \sin\theta$.

On the other hand, for $\Vert t\Vert\le \frac{2\pi}{a\g}$, when $\ell\ge \g$ and thua $\sigma:=a\g\Vert t\Vert\le s$,
\begin{align*}   
h(s,\ell)&= \int_0^\infty\frac{e^{-u}\dd u}{\sqrt[4]{1+s^2u^2}}
\le 
\int_0^\infty\frac{e^{-u}\dd u}{\sqrt[4]{1+\sigma ^2u^2}}
\le 
\int_0^1\frac{e^{-u}\dd u}{\sqrt[4]{1+\sigma^2u^2}}
+\int_1^\infty\frac{e^{-u}\dd u}{\sqrt[4]{1+\sigma^2u^2}}
\\ &
\le
\int_0^1\left(1-\frac{\sigma^2u^2}{50}\right) e^{-u}\dd u
+\int_1^\infty e^{-u}\dd u
\\
&=
-\int_0^1\frac{\sigma^2u^2}{50} e^{-u}\dd u
+\int_{0}^\infty e^{-u}\dd u=1-0.0016... \sigma^2
\le \exp\left\{ -\frac{\sigma^2}{700}\right\}=\exp\left\{ -\frac{\g^2 \Vert t\Vert^2 a^2}{700}\right\},
\end{align*}
since $(1+x^2)^{-1/4}\le 1-x^2/50$ for $0\le x\le 7$, and $\sigma\le 2\pi<7$ by assumption.
By iterating the same argument as before, we get
$$
\varphi_{\tilde W_n}\le h(s,\ell)^{n/4} \le \exp\left\{ -\frac{n \g^2 \Vert t\Vert^2 a^2}{2800}\right\}.
$$
Consequently,
\begin{align}\label{bound2}
\iint_{\Vert t\Vert\le \frac{2\pi}{a\g}}
\left|\varphi_{\tilde W_n}(t) \right|\dd t_1 \dd t_2 
& \le 
\iint_{\Vert t\Vert\le \frac{2\pi}{a\g}}
\exp\left\{ -\frac{n \g^2\Vert t\Vert^2 a^2}{2800}\right\}
\dd t_1 \dd t_2 
\\ & \le 
\frac{1400}{a^2\g^2}\int_0^{2\pi}d\theta
\int_0^\infty 
\exp\left\{ -\frac{nr^2}2\right\}r\dd r
 =
\frac{2800\pi}{n a^2\g^2}\nonumber
\end{align}
by changing the variables $t_1=\frac{10\sqrt{14}}{a\g} r \cos\theta$, $t_2=\frac{10\sqrt{14}}{a\g} r \sin\theta$.

Finally, recalling that $a_1\sim n^\frac{\a}{1-\a}$, for $\hat W_t^{(1)}$ and $a_2\sim (\ln{n})^\beta$, for $\hat W_t^{(2)}$, from~\eqref{eqFT}, \eqref{bound1} and~\eqref{bound2} we obtain that
\begin{align*}
f_{\tilde W_n^{(1)}}(x,y)&\le O\left(\frac{1}{n^{1+\frac{2\a}{1-\a}}}\right);\\
f_{\tilde W_n^{(2)}}(x,y)&\le O\left(\frac{1}{n(\ln{n})^{2\beta}}\right).
\end{align*}
Hence
\begin{align*}
\P(\hat W_n\in R_k)&\le \P(\mathcal{E}_1^c)+
\P(\mathcal{D}_2^c)+
\P(\tilde W_n\in R_k)= 4e^{-n/36}+
\iint_{(x,y)\in R_k} f_{\tilde W_n}(x,y)\dd x \dd y
\\ &
=2\pi k\times
\begin{cases}
O\left(\frac{1}{n^{1+\frac{2\a}{1-\a}}}\right),
&\text{for }W_n^{(1)};\\
O\left(\frac{1}{n(\ln{n})^{2\beta}}\right),   
&\text{for }W_n^{(2)}
\end{cases}  
\end{align*}
since the area of $R_k$ is $\pi(2k+1)$.
\end{proof}

\begin{lemma}\label{lemCV2}
Suppose $d\ge 2$, and $\hat W_n=r>1$. Then 
$$
\P\left(\Vert\hat W(t)\Vert\le 1 \text{ for some }t\in[\tau_n,\tau_{n+1}\right)\le \frac{1}{2r}.
$$
\end{lemma}
\begin{proof}
First, note that $\hat W_n$ lies outside of the unit circle on $\R^2$. The projection $\tilde{\bf f}_{n+1}$ of ${\bf f}_{n+1}$ on the first two coordinates plane has an angle $\phi$ uniformly distributed over $[0,2\pi]$. The probability in the statement of the lemma is monotone increasing in $\tau_{n+1}$, hence it is bounded above by
$$
\P(\text{the infinite ray from }\hat W_n\text{ in the direction  }\tilde{\bf f}_{n+1}\text{ passes through the unit circle})=\frac{2\arcsin(1/r)}{2\pi}
$$
(see Figure~\ref{fig:ray}.) Finally, $\arcsin(x)\le \frac{\pi x}{2}$ for $0\le x\le 1$.
\end{proof}

\begin{figure}
\begin{center}
\setlength{\unitlength}{1mm}
\begin{picture}(100,16)
\put(20,8){\circle{14}}
\put(20,8){\color{blue}\line(1,3){2.2}}
\put(20,8){\color{blue}\line(1,-3){2.2}}
\put(42.5,8){\line(-3,1){30}}
\put(42.5,8){\line(-3,-1){30}}
\put(42.5,8){\color{blue}\line(-1,0){22.5}}
\put(42.5,8){\color{red}\vector(-4,-1){40}}
\put(41,13){$\arcsin(1/r)$}
\put(40,13){\vector(-1,-1){4}}
\end{picture}   
\end{center}
    \caption{A ray from $\hat W_n$ passes through the unit circle}
    \label{fig:ray}
\end{figure}
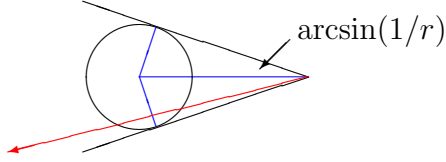

\begin{thm}\label{thm:CV}
Let $d\ge 2$ and the walk $Z(t)$ has either the rate~\eqref{rate1} with $\a\in(0,1)$ or  rate~\eqref{rate2} with $\beta>2$. Then $Z(t)$ is transient a.s.
\end{thm}
\begin{proof}
As mentioned before, we will only show that the walk is not $\rho-$recurrent for $\rho=1$. To do that, it will suffice  that a.s.\ there will be only finitely many $n$ such that the event
\begin{align*}
A_n&=\left\{\Vert \hat{W}(t)\Vert\le 1\text{ for some }t\in [\tau_n,\tau_{n+1}]\right\}    
\end{align*}
occurs. Indeed, fix a positive integer $n$. At this time, 
$\hat W_n\in R_k$ for some $k\in \{0,1,2,\dots\}$.
According to Lemma~\ref{lemCV2},
$$
\P(A_n\mid \hat W_n\in R_k)\le \begin{cases}
    1,&\text{if }k=0;\\
    \frac1{2k}, &\text{if }k\ge 1.
\end{cases}
$$
Fix some very small $\eps>0$. Then
\begin{align*}
\P(A_n)&\le \sum_{k=1}^{n^{\frac{\a}{1-\a}+\eps}}
\P\left(A_n\mid \hat W_n^{(1)}\in R_k\right)\P\left( \hat W_n^{(1)}\in R_k\right)+
\P\left(A_n\mid \Vert \hat W_n^{(1)}\Vert\ge n^{\frac{\a}{1-\a}+\eps}\right)
\\
&\le \frac{C n^{ \frac{\a}{1-\a}+\eps  }}{n^{ 1+\frac{2\a}{1-\a}}}+ 
\P\left(\tau_{n+1}-\tau_n\ge n^{\frac{\a}{1-\a}+\eps}-1\right)
\le \frac{C}{n^{ 1+\frac{\a}{1-\a}-\eps  }}+ 3 e^{-c_1^* n^\eps}
\end{align*}
by~\eqref{eqnotlarge} and Lemma~\ref{lemCV1}.
Hence $\sum_n \P(A_n)<\infty$ and the result follows from the Borel-Cantelli lemma.

Similarly, in the other case
\begin{align*}
\P(A_n)&\le \sum_{k=1}^{2(\ln{n})^{1+\beta}}
\P(A_n, \hat W_n^{(2)}\in R_k)+
\P\left(A_n\mid \Vert \hat W_n^{(2)}\Vert\ge {2(\ln{n})^{1+\beta}}\right)
\\
&\le \frac{ 2C(\ln{n})^{1+\beta}  }{n(\ln{n})^{2\beta}}+ 
\P\left(\tau_{n+1}-\tau_n\ge 2(\ln{n})^{1+\beta}-1\right)
\le \frac{2C}{n (\ln{n})^{\beta-1}  }+ \frac{1}{n^2}
\end{align*}
for all sufficiently large $n$ by~\eqref{eqnotlarge2} and~Lemma~\ref{lemCV1}.
Hence $\sum_n \P(A_n)<\infty$ for $\beta>2$, and the result again follows  from the Borel-Cantelli lemma.
\end{proof}

\section*{Appendix: properties of Bessel function $J_0$}
Let
$$
J_0(x)=\sum_{m=0}^\infty \frac{(-1)^m\,(x/2)^{2m}}{(m!)^2}
$$
be the Bessel function of the first kind. The following statement might be known, but we could not find it in the literature.
\begin{claim}\label{claimJ}
 For all $x\ge 0$
    \begin{align}\label{Besbound2}
    |J_0(x)|\le \frac1{\sqrt[4]{1+x^2}}=:G(x).
\end{align}
\end{claim}
\begin{proof}
From Theorem~1 and formula (1) in~\cite{Krasikov} we get (with $\nu=0$, $\mu=3$)
\begin{align*}
    J_0^2(x)&\le \frac{4(4x^2-5)}{\pi( (4x^2-3)^{3/2}-3  )},\quad x\ge 1.13.
\end{align*}
This inequality also implies that
$$
|J_0(x)|\le \frac1{\sqrt[4]{1+x^2}},\quad x\ge 1.13.
$$
At the same time for $0\le x\le 2$ (using \cite{AS} 9.1.14)
$$
J_0^2(x)=\sum_{k=0}^\infty (-1)^k\frac{(2k)!}{(k!)^4} \left(\frac{x}{2}\right)^{2k}
\le 1 - 2\left(\frac{x}{2}\right)^2 + \frac32 \left(\frac{x}{2}\right)^4 - \frac59 \left(\frac{x}{2}\right)^6 + \frac{35}{288}\left(\frac{x}{2}\right)^8
$$
and at the same time
$$
\frac{1}{\sqrt{1+x^2}}\ge 
1 - 2\left(\frac{x}{2}\right)^2 + \frac32 \left(\frac{x}{2}\right)^4 - \frac59 \left(\frac{x}{2}\right)^6 + \frac{35}{288}\left(\frac{x}{2}\right)^8
$$
yielding~\ref{Besbound2}.
\end{proof}
Observe as well that $0\le G(x)\le 1/\sqrt{x}$ and that $G$ is a decreasing function for $x\ge 0$; see Figure ~\ref{fig:Bes}
\begin{figure}
    \centering
    \includegraphics[scale=0.2]{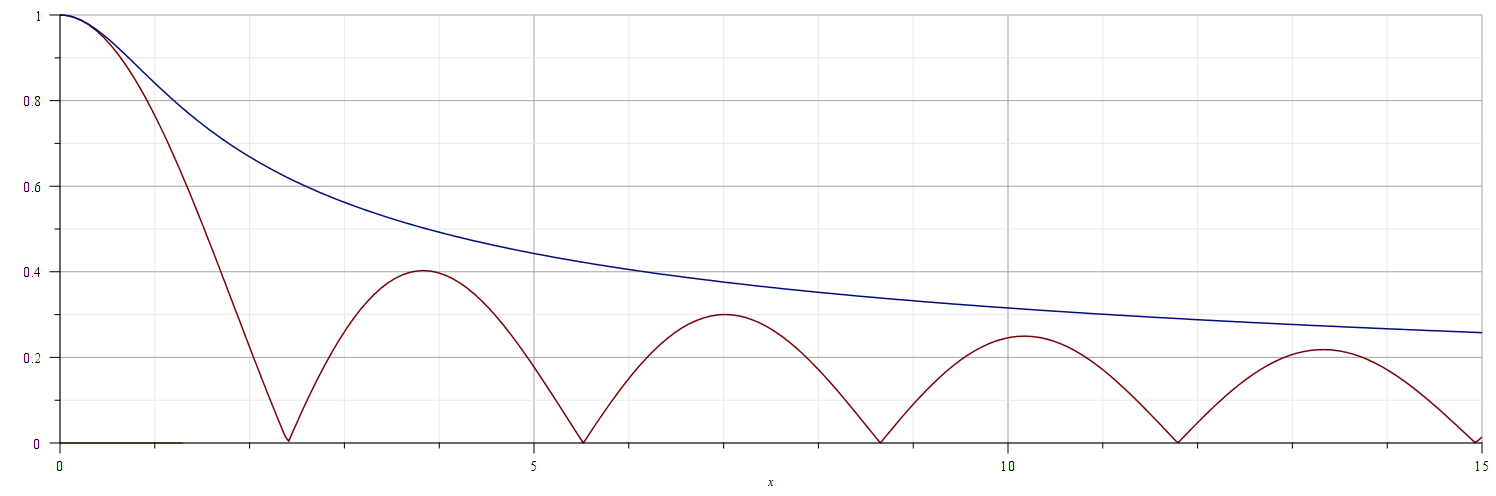}
    \caption{$|J_0(x)|$ (red) and its upper bound  $G(x)$ (blue).}
    \label{fig:Bes}
\end{figure}

\section*{Acknowledgement}
The research is supported by Swedish Research Council grant VR 2014-5157 and Crafoord Foundation grant 20190667.

\end{document}